\documentclass[a4paper,twoside,english]{amsart}
\usepackage[T1]{fontenc}
\usepackage[utf8]{inputenc}
\usepackage{babel}
\usepackage{verbatim}
\usepackage[normalem]{ulem} 
\usepackage{enumitem}
\usepackage{amsmath,amsthm,amssymb}

\usepackage[unicode=true,
 bookmarks=true,bookmarksnumbered=true,bookmarksopen=false,
 breaklinks=false,pdfborder={0 0 1},colorlinks, citecolor=red,
        backref, pdfauthor={Francesco Veneziano, Umberto Zannier},
        pdftitle={On the Disk of Convergence of Algebraic Power Series}]
 {hyperref}

\usepackage{wrapfig}
\usepackage[width=.75\textwidth]{caption}

 \usepackage[alphabetic, backrefs]{amsrefs} %

\makeatletter

\pdfpageheight\paperheight 
\pdfpagewidth\paperwidth

\theoremstyle{plain}
\newtheorem{thm}{\protect\theoremname}[section]
 \newcommand\thmsname{\protect\theoremname}
 \newcommand\nm@thmtype{theorem}
 \theoremstyle{plain}

  \theoremstyle{remark}
  \newtheorem{rem}[thm]{\protect\remarkname}
  \theoremstyle{definition}
  \newtheorem*{example*}{\protect\examplename}
  \theoremstyle{definition}
  \newtheorem{example}[thm]{\protect\examplename}
  \theoremstyle{plain}
  \newtheorem{lem}[thm]{\protect\lemmaname}
  \theoremstyle{plain}
  \newtheorem{prop}[thm]{\protect\propositionname}
  \theoremstyle{plain}
  \newtheorem{cor}[thm]{\protect\corollaryname}

 \newtheorem*{question}{Question}
 \newtheorem*{openquestion}{Open Question}
\makeatletter
  \theoremstyle{definition}
  \newtheorem{my@rem}[thm]{Remark}
  \renewenvironment{rem}{\begin{my@rem}}{\end{my@rem}}
\makeatother

\makeatother

  \providecommand{\examplename}{Example}
  \providecommand{\lemmaname}{Lemma}
  \providecommand{\propositionname}{Proposition}
  \providecommand{\remarkname}{Remark}
  \providecommand{\theoremname}{Theorem}
\providecommand{\theoremname}{Theorem}
 \providecommand{\corollaryname}{Corollary}

\usepackage[left=3.6cm,right=3.6cm,top=3cm,bottom=3cm]{geometry}

\usepackage{mathtools}
\usepackage{xcolor}

\newcommand{\C}{\mathbb{C}}
\newcommand{\R}{\mathbb{R}}
\newcommand{\Q}{\mathbb{Q}}
\newcommand{\Z}{\mathbb{Z}}
\newcommand{\N}{\mathbb{N}}
\newcommand{\A}{\mathbb{A}}
\renewcommand{\P}{\mathbb{P}}

\def\O{{\mathcal O}}

\def\d{{\rm d}}

\def\CVD{{\hfill\hfil{\lower 2pt\hbox{\vrule\vbox to 7pt
{\hrule width  4pt\varphifill\hrule}\varphirule}}}\par}

\newcommand{\abs}[1]{\left| #1 \right|}
\newcommand{\floor}[1]{\left\lfloor #1 \right\rfloor}
\newcommand{\ceiling}[1]{\left\lceil #1 \right\rceil}

\title{On the disk of convergence of algebraic power series}

\author{Francesco Veneziano}
\address{Department of mathematics, University of Genova, Via Dodecaneso~35, 16146 Ge\-no\-va, Italy}
\email{francesco.veneziano@unige.it}

\author{Umberto Zannier}
\address{Scuola Normale Superiore, Piazza dei Cavalieri 7, 56126 Pisa, Italy}
\email{umberto.zannier@sns.it}


\begin{document}

\begin{abstract}
This paper is mainly concerned with the  disk of convergence of a power series $s(x)$ representing an algebraic function of $x$ and specifically with the relation between this disk and the branch points of the function. We shall focus especially on the $p$-adic case,  answering some questions of basic nature, seemingly absent from the existing literature. Our methods are simple and essentially self-contained.

To illustrate the issues, recall that in the complex case it follows from standard arguments that the open disk of convergence cannot contain all the branch points of $x$ unless the series represents a rational function.
In the $p$-adic case, we show that the analogous assertion is not true in complete generality; but we also confirm it in a number of cases, for instance under the assumption that $p$ is not smaller than the degree of $s(x)$ over the field of rational functions of $x$.  In particular this gives an upper bound for the radius of convergence which has intrinsic nature.

We shall also touch several related questions.
\end{abstract}

\maketitle

\section{Introduction}
The present paper is motivated by a natural question about {\it algebraic} power series, which is simply stated and might maybe look trivial at first sight, but does not seem to be treated explicitly in the existing literature, at any rate to our knowledge.

The  issue concerns the {\it disk  of convergence} of a power series $s(x)=\sum_{m=0}^\infty a_mx^m\in K[[x]]$ representing an `algebraic function' of $x$, i.e.,   such that $f(x,s(x))=0$ for some nonzero polynomial $f\in K[X,Y]$. Here  $K$ is a   field   of characteristic $0$  equipped with a (nontrivial) absolute value $| \cdot |$.   It is well known (and not difficult to prove) that  $s(x)$ has a nonzero radius of convergence with respect to this absolute value.\footnote{We refer specifically to the paper \cite{DR}, which will be more relevant later, for proofs valid also in the ultrametric case.}

\medskip

Such radius of convergence has been studied for a long time, both in the classical case over $\C$ and for fields equipped with an ultrametric absolute value, which is the main setting for this paper. In this last case the radius of convergence has also an arithmetical significance, for example in connection with the famous theorem of Eisenstein on the coefficients of an algebraic series defined over a number field, stating that  for   suitable integers $b, b'>0$ all  the numbers $b'b^{n}a_n$ are algebraic integers.  This qualitative result has been made explicit and refined by subtle $p$-adic analysis in the papers \cite{DR} by Dwork and Robba and \cite{DvdP} by Dwork and van der Poorten; the issue was further studied e.g. by  Bilu and Borichev in  \cite{BB} and by Mechik  in \cite{M1}, \cite{M2}.  More recent investigations  study the radius from different viewpoints; for instance the paper \cite{B} by Baldassarri  proves  the continuity of the radius (depending on the center of the expansion),  in fact for a rather more general class of functions satisfying linear differential equations.

\medskip

The papers mentioned above give `explicit' estimates for the radius of convergence, leading to bounds for the so-called {\it Eisenstein constant} (related to the  integer $b$ above).

Here instead we shall  focus on the {\it relation between the radius and the branch points}  of the coordinate $x$ on the curve $f=0$.\footnote{We recall below  in \S~\ref{S.notation} some relevant definitions.}

Together with the poles,  the branch points make up {\it all} potential singularities for $s(x)$ (which is a well-defined analytic function of $x$ in any open disk of convergence). When $s(x)$ represents a rational function, there are no branch points and the only possible singularities are poles.
Then,  in the converse direction the following  question arises:

\medskip

\begin{question}
 Assuming that $s(x)$  does not represent a rational function, is it possible that it converges in an open disk  
containing all the (finite) branch points of $x$?
\end{question}

\medskip

Let us remark at once that, naively,  we would expect a \emph{negative} answer: as mentioned,  any possible singularity preventing the extension of $s(x)$ has to be sought among {\it poles} and {\it branch points}. Here the second type is excluded by assumption (for {\it finite} branch points), whereas the polar singularities may be eliminated by the simple device of multiplying $s(x)$ by a suitable nonzero polynomial.
Thus it might seem  at first sight that we may eliminate all (finite) singularities, which would imply that $s(x)$ is a polynomial.

\begin{wrapfigure}{r}{0.3\textwidth}
\includegraphics[width=0.28\textwidth]{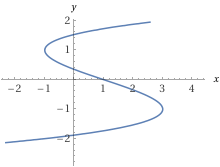}
\centering
\caption{\tiny Plot of the curve $x = y^3 - 3 y + 1$. Over $\C$ the disk of convergence of the series $y(x)$ with negative value at 0 has a radius of 3 and contains the branch point at $x=-1$.}\label{fig:1}
\end{wrapfigure}

However, this sketch of approach needs details and further consideration to be be carried out; for instance  $s(x)$ cannot generally be extended at once  to a global function with the sought  properties on the whole of $K$ even after eliminating poles.

This kind of analysis is especially delicate  for ultrametric absolute values.
We also note that, even over $\C$, it may well happen that the disk of convergence contains {\it some} branch point, corresponding to {\it branches} of the algebraic function different from the one represented by the series $s(x)$ in the disk; a basic example of the situation is depicted in Figure~\ref{fig:1}. Our concern is with disks containing {\it all} branch points.

\medskip

The present question  looks (and partly \textit{is}) elementary, so we hope it will not be entirely  free of interest  even  for possible readers
not fully familiar  with the context. Especially thinking of  them, we found  it worthwhile to recall in \S~\ref{S.notation} some facts from the basic theory of algebraic functions. We stress that our methods are essentially elementary and almost self-contained.

In the same section we will also fix some notation, terminology, and normalizations. But, before that, let us state our main results.

\subsection{Setting and main statements}
Field extensions will not affect our question, so we shall suppose that $K$ is algebraically closed of characteristic $0$ and complete with respect to a nontrivial absolute value $|\cdot |$.
We denote by $d>0$ the degree of $s(x)$ over $K(x)$, so $f(x,s(x))=0$ for some nonzero irreducible polynomial $f\in K[X,Y]$, of degree $d$ in $Y$.

\medskip

{\bf The complex case}.  The complex case (i.e.,  $K=\C$)  is  easy to study, and a {\it negative}  answer to the  Question has been somewhat well-known for centuries.  It may be proved essentially by analytic continuation and, in fact, it is to some extent   implicit in most basic  treatments of the theory of compact Riemann Surfaces. However, for the sake of completeness  we shall recall a couple of arguments in detail, since we did not find it easy to locate an explicit statement in the literature.

 \medskip

 {\bf The ultrametric case}. On the other hand, if the absolute value of $K$ is ultrametric  (for instance if $K$ is the completion $\C_p$ of an algebraic closure of $\Q_p$) analytic continuation is not possible and the arguments mentioned above for $\C$ fail in a crucial way.

 In the present paper we shall see by simple examples  that, indeed, given any prime number $p$,

\smallskip
 \noindent{\it the expected (negative) answer to the Question  does  {\bf not}  generally hold  over $K=\C_p$.}
\smallskip

 This may be restated as the assertion that  {\it there exist  non-rational algebraic functions represented by a power series which converges in a disk containing all the (finite) branch points}.

However,  by essentially elementary means we shall also prove, under certain additional hypotheses, a {\it negative} answer to the Question in the ultrametric case. 
 For instance, relying on a certain  theorem of Dwork and Robba \cite{DR}, we shall  observe that  
 {\it  the Question has a negative  answer if the residue field has characteristic $p\ge d$} (the case $p=d$ is slightly more delicate and requires to enter into the proofs of \cite{DR}).  We  shall also notice that the same conclusion holds under simple additional assumptions on the ramification.

 \medskip

 More precisely, with the notation as above, we have the following results.

  \begin{thm} \label{T.Y} Let $d:=[K(x,s(x)):K(x)]$ and suppose that one at least of the following conditions is satisfied:
\begin{enumerate}[label = (\roman*)]
 \item\label{T.Y:i} $K=\C$.\footnote{We could merely suppose that $K$ contains $\C$ inducing on it the usual absolute value. But it is known that then $K=\C$, see for instance Corollary~2.4, \S~2, Ch.~XII in Lang's Algebra book \cite{LangAlgebra}.}

 \item\label{T.Y:ii} The absolute value is ultrametric, with residue field of characteristic $0$ or $p\ge d$.
 \item\label{T.Y:iii}  There exists $x_0\in K$ such that  all the ramification indices above $x_0$ are greater than  $1$.
\end{enumerate}
    
Then the Question has a negative answer. Namely, if the series $s(x)\in K[[x]]$   is convergent in a  disk $\{z\in K: |z|< R\}$ containing every finite branch point of $x$, then  $s(x)$ is a rational function (i.e. $d=1$).
 \end{thm}

Note that, in particular, condition \ref{T.Y:iii} holds if   the extension $K(x,s(x))/K(x)$%
is  Galois   of degree $d>1$, since in such an extension  the ramification indices above a given branch points are all equal. Therefore a negative answer to the Question holds for $d\le 2$ without supplementary assumptions.
On the other hand we shall see that the inequality $d\leq p$ in (ii) cannot be generally improved.

 \medskip

{\bf Converse direction and an open question}. Suppose for simplicity that  the polynomial $f(X,Y)$ is monic in $Y$, which can be achieved by multiplying $s(x)$ by the leading coefficient of $f$; this has the effect of removing all polar singularities of $s(x)$. Then in the complex case, the argument by analytic continuation shows that the disk of convergence goes {\it at least} up to the {\it nonzero} branch point {\it nearest} to the origin; therefore the closure of the disk of convergence of $s(x)$ contains at least one branch point different from the origin. This conclusion goes somewhat in the converse direction of our results.

In the ultrametric case the analogous conclusion is not generally true. The issue has been studied in several works, also because of applications to explicit versions of  the already mentioned classical Eisenstein theorem on algebraic power series, in the shape  that only finitely many primes appear in the denominators of the coefficients of an algebraic power series defined over a number field.
 The paper of Dwork and Robba \cite{DR} yields  sufficient conditions  which are sometimes best-possible, and this is again taken into account with additional precision in \cite{DvdP}. The  paper \cite{DR} shall be  crucial also for  our proof of part \ref{T.Y:ii} of Theorem \ref{T.Y}, and we shall enter into those proofs for treating the `borderline' case $p=d$.
 Also, combining \cite{DR} with Theorem \ref{T.Y}\ref{T.Y:ii}, we shall immediately obtain the following further result (where we normalize  $f$ as above).

 \begin{cor} \label{C.Yii} Suppose that $f(X,Y)$ is monic in $Y$, of degree $d>1$, and that either $K=\C$ or that the absolute value is ultrametric, with residue field of characteristic $0$ or $p> d$.

 Then,  there exist finite  nonzero branch points $\rho_1,\rho_2\in K^*$ such that the radius of convergence $R$  satisfies 
 $|\rho_1|\le R\le|\rho_2|$.
 \end{cor}

Observe that for this Corollary, in the ultrametric case  we need the stronger assumption that $p>d$. In Theorem  \ref{T.N}\ref{T.N:ii}  below we point out very simply that this cannot be relaxed to $p\ge d$ (contrary to Theorem \ref{T.Y}).
 
\medskip

These results give information on the radius of convergence of different flavour compared to past results mentioned above.

In the case $K=\C$ the arguments will show that actually the radius $R$ {\it equals}  the absolute value of some branch point. For ultrametric valuations, in the most general case, we do not know whether this is true.
Therefore we may ask whether a sharpening of the Corollary holds, and formulate this as follows:

\medskip

\begin{openquestion}In the case of an ultrametric valuation, under the assumptions of Corollary~\ref{C.Yii} (in particular, $p>d$), is it true that the radius $R$ {\it equals}  the absolute value of some branch point?
\end{openquestion}

 At first sight this looks as an attractive question to us, as it would give a meaningful geometric interpretation of the radius of convergence.\footnote{It may be that this question admits an easy anwer. We wonder whether it might be approached via some extension to annuli, in place of disks, of Theorem~2.1 of \cite{DR}, also used below as Proposition~\ref{P.DR}. However we do not know whether such an analogue is sensible, not to mention true.}

 \medskip

If the radius of convergence is equal to the distance to some branch point, it follows that it is of the form $p^\beta$ for some \emph{rational} value of $\beta$. We plan to come back to this topic in a future note, establishing that this conclusion is anyways true. Compare this behaviour with what happens in the more general case of series satisfying differential equations, where some statement in this direction can be obtained, but not in full strength (see \cite{Kedlaya_2010}*{Corollary 11.8.2} and our Remark~\ref{rem:ex_ogni_raggio} later).

 \bigskip

{\bf About counterexamples}.  As said, the condition $p\ge d$ in Theorem \ref{T.Y}\ref{T.Y:ii}, i.e.  in the cases when the absolute value is ultrametric with residue field of characteristic $p>0$, cannot generally be removed and  plays a crucial role in the theorem of Dwork and Robba \cite{DR} that we shall use. However {\it a priori} it does not seem obvious that it cannot be removed in our application as well.\footnote{Indeed, the (counter)examples given  in \cite{DR} do not apply to our situation.}
 In fact, this  turns out to be  the case.

 Also, the condition $p>d$ cannot be relaxed to $p\ge d$ in  Corollary \ref{C.Yii}, this time with more evident examples (appearing already in \cite{DR}). We state  this as the following:

 \begin{thm}[Counterexamples] \label{T.N}   \leavevmode
 \begin{enumerate}[label = (\roman*)]
  \item\label{T.N:i} For each prime $p$ there exists a non-rational algebraic  power series as above, with $K=\C_p$, $d=p+1$, converging in a disk $\{z\in K: |z|< R\}$
 which contains all finite branch points.
  \item\label{T.N:ii} For each prime $p$, there exists an algebraic series as above of degree $p$ over $\C_p(x)$ whose radius of convergence is strictly smaller than the minimal absolute value of a nonzero finite branch point.
 \end{enumerate}
 \end{thm}
With hindsight these counterexamples are certainly neither deep nor unexpected, but we could not find them in the literature and in fact they are apparently not widely known.

 \begin{rem}\label{R.diff}  {\bf Linear differential equations}. We remark that algebraic functions in $K[[x]]$ satisfy linear differential equations over $K(x)$ of the so-called {\it Fuchsian type}, namely having only regular singularities; see \cite{DGS} for the basic theory, also over $p$-adic fields. In the case of algebraic functions, the branch points appear among the singularities of such differential equations. One could then formulate a similar {\it  Question} for solutions of differential equations of Fuchsian type.  We do not have any result in  such  direction. We note that the mentioned paper \cite{DR}, although using linear differential equations for algebraic functions, does not seem to contain analogous results of similar strength, nor remarks in the same direction, for solutions more general than the algebraic ones (but see \cite{DGS}). (We  further observe that the condition that the  singularities are {\it regular}, and  also that the {\it  exponents} are rational, are highly relevant here; indeed,  see the paper \cite{BS} of  Bombieri-Sperber for some upper bounds on the radii of convergence when these conditions are not fulfilled.)
 \end{rem}

\subsection{Further related issues}  Note that a negative answer to the Question for a certain $K,s(x)$  may be restated (in the contrapositive sense) as follows:

{\it If the series $s(x)$ converges in a disk $\{z\in K:|z|<R\}$, then there exists a branch point $\rho\in K$ such that $|\rho|\ge R$}.

Here a further question arises about the behaviour of the series on the boundary of the disk of convergence. Suppose that the series $s(x)$ converges in the `closed' disk\footnote{We use the quotation marks because in the ultrametric case both kinds of disks are  open and closed in the topology, so that `closed' refers just to the shape of the disk.}   $\{z\in K:|z|\le R\}$; 
then it seems reasonable to ask if this implies the existence of a branch point $\rho\in K$ satisfying the strict inequality $|\rho|>R$.

Note that, leaving aside the easy case of rational functions,
such (absolute) convergence on the boundary might well happen in the complex case with plenty of simple examples: 
consider for instance the expansion  $$(1+x)^\frac{1}{2}=\sum_{n=0}^\infty \binom{1/2}{n}x^n=\sum_{n=0}^\infty (-1)^{n-1} \binom{2n}{n} \frac{x^n}{2^{2n}(2n-1)};$$ here the $n$-th coefficient is in absolute value $\asymp n^{-3/2}$, so there is absolute convergence for $|x|\le 1$, though the radius of convergence is precisely $1$.

 This example also shows that for $K=\C$ (and $d\ge 2$)  it can happen that {\it all}  finite branch points lie in the closed disk of convergence (in the example, $-1$  is the only finite branch point).

On the contrary, it is a (special case of a) theorem due to Bosch-Dwork-Robba \cite{BDR} that, for ultrametric valuations, if an algebraic series $s(x)$ converges in a disk $\{z\in K:|z|\le R\}$, where $R$ is in the value group $|K^*|=\{|x|:x\in K^*\}\subset \R_{> 0}$,  then in fact its radius of convergence must be strictly larger than $R$.

This implies the following refinement of part \ref{T.Y:ii} of Theorem \ref{T.Y}.
\medskip

\noindent{\bf Refinement of Theorem~\ref{T.Y}\ref{T.Y:ii}}: {\it
 In the above assumptions (ii),  if the series $s(x)$ is not rational and converges in a closed disk $\{z\in K:|z|\le R\}$ then there exists a branch point $\rho \in K^*$ such that $|\rho|> R$}.

\medskip
Since this issue appears naturally in the present context, we state explicitly the relevant  result of the paper \cite{BDR}, giving  for it a self-contained independent proof; namely, we shall prove the following statement

\begin{thm}\label{T.BDR} \cite{BDR}  Notation being as above, assume that $f\in K[X,Y]$ for an ultrametric, algebraically closed, complete field $K$ of characteristic $0$,   and assume  that the algebraic series $s(x)$   converges in the closed disk $\{x\in K:|x|\le r\}$, where $r$ is in the value group. Then the radius of convergence of $s(x)$ is strictly larger than $r$, i.e. there exists $r'\in\R$, $r'>r$,  such that $s(x)$ converges in the disk $\{x\in K:|x|<r'\}$.
\end{thm}

This result appears in \cite{BDR}, actually in a more general form\footnote{That context admits  several variables and coefficients that are themselves series convergent in a larger disk.}  and with a different formulation; see also \cite{C}, Prop. 4.2, for a variation.
We shall provide a different and less demanding proof
with respect to \cite{BDR}.

\medskip

This theorem implicitly provides the following

{\bf Sufficient criterion for the transcendence of the series $s(x)$}: 

 {\it If the coefficients $a_n$ are such that $v_p(a_n)\to \infty$ but $\liminf v_p(a_n)/n=0$ then the series $\sum_{n=0}^\infty a_nx^n$ cannot be algebraic over $K(x)$}. 
 
\medskip

Indeed, the conditions imply that $s(x)$  converges for $\abs{x}_p\leq 1$ but does not converge in any  disk with strictly larger radius, and therefore it must be transcendental. 

The topic of criteria for the transcendence of such series is itself an interesting subject, also recently investigated, e.g.  by Bostan, Salvy, Singer in \cite{BSS}; in this paper  they offer  algorithms for deciding whether certain series are transcendental, and also a wide variety of interesting criteria.
The present criterion seems to escape from those, though admittedly it offers a  limited range of possible applications.

\medskip

Puzzled by this latter kind of results, we wondered about what happens in the case of solutions of Fuchsian linear differential equations (a class which includes algebraic functions).  Again, we have no reference on an analysis of this  simple question; apparently there are results concerning {\it analytic elements}  (see  for instance \cite{Dw}) but we have no knowledge of similar conclusions concerning the case of  analytic power series. 

For solutions of Fuchsian equations,  we have noticed some simple counterexamples to a conclusion as in Theorem \ref{T.BDR}, which we state as the following theorem, limiting to the case $p=2$ for simplicity, and denoting by $|\cdot |_2$ the standard $2$-adic absolute value for this statement.

\begin{thm}\label{T.N2}  There exist $2$-adic integers  $\gamma_1,\gamma_2\in\Z_2$ such that putting $a_m:=\binom{\gamma_1}{m}\binom{\gamma_2}{m} \in\Z_2$, we have $\left |a_m\right |_2\to 0$   as $m\to\infty$, but $\left |a_{m}\right |_2\gg m^{-1}$ for infinitely many $m$.

The series $s(x):=\sum_{m=0}^\infty a_mx^m$ satisfies the differential equation
$$x(1-x) s''(x) + (1+(\gamma_1+\gamma_2-1)x)s'(x)-\gamma_1 \gamma_2 s(x)=0,$$
corresponding to the linear recurrence
$(m+1)^2a_{m+1}=(\gamma_1-m)(\gamma_2-m)a_m.$
Furthermore, the series $s(x)$ and all its powers $s(x)^q$ converge in the $2$-adic `closed' disk $\{x\in\C_2:|x|_2\le 1\}$ but in no strictly larger disk.
\end{thm}

\medskip

Note that the assertion about the powers $s(x)^q$ is indeed stronger than for the mere series $s(x)$, since the radius of convergence can increase upon taking a power, as shown by examples like $(1+x)^{1/p}$ over $\C_p$. (In the present case  only the exponents $q=2^h$ are significant, since one can take roots of odd order over $\C_2$ without introducing denominators.)


\medskip

\begin{rem}
We remark that the series $s(x)$ of Theorem~\ref{T.N2} is the expansion of the hypergeometric function ${}_2F_1(-\gamma_1,-\gamma_2;1;z)$.
Such series
are Fuchsian by general theory, and the exponents are easily computed in terms of $\gamma_1,\gamma_2$:  see \cite{DGS}, Ch IV. 

The proof of Theorem \ref{T.N2} will show that there is plenty of freedom in the choice of this counterexample, and the construction can be easily modified and extended. The $2$-adic integers   $\gamma_1,\gamma_2$  in question, provided by the proof,  are most probably transcendental, so they do not yield `global' examples; to prove this transcendence appears  at first sight to lie outside the presently known techniques.
{\it A fortiori},  it seems very hard to prove independently that there are not such examples over $\overline\Q$, since implicitly this would provide such a transcendence proof.

We also note that  the proof allows $\gamma_1,\gamma_2$ with the required properties to be chosen algebraically independent. However by specializing them to algebraic numbers near them will not generally preserve the property. Hence there is no continuity for such a kind of convergence, differently for instance to what happens for the radius (as a function of the center), as proved in \cite{B}.
\end{rem}

\medskip

Related (but different) functions have appeared previously in connection with the subject of $p$-adic convergence, in particular on the boundary of the disk of convergence. See for example \cite{Katz}, p. 74, where a ratio of hypergeometric functions $F(1/2, 1/2, 1, t)/F(1/2, 1/2, 1, t^p)$ is considered;  a corresponding series converges on a nonempty open subset of the boundary of the disk of convergence, related to supersingular values of the Legendre parameter (relative to the prime $p$). This is significant in relation to the monodromy at infinity.
  In particular it turns out that the corresponding series  does not converge on the full boundary.

\medskip

\medskip

One might ask whether a similar example exists for linear differential equations of order $1$, i.e. $s'(x)=u(x)s(x)$, where $u\in K(x)$.  See \S~\ref{R.N2} after the proof of Theorem \ref{T.N2} for a brief discussion, and Proposition \ref{P.N2} therein, concluding in particular that
no such example exists:
this is an analogue of Theorem \ref{T.BDR} for the series in question.

 \medskip
 
 \noindent{\bf Acknowledgements}. We thank Yves Andr\'e, Francesco Baldassarri, Nickolas M. Katz, and Hector Pasten, for their kind attention, comments, and references.

\medskip

\section{Some notation and background} \label{S.notation} For the sake of completeness and full clarity, we now explain in more detail and self-contained terms  the meaning of our problems and assertions.

As a reference for basic properties of algebraic functions, we can indicate e.g. Lang's book \cite{L}.
For basic and deeper facts on $p$-adic valuations, norms and $p$-adic analysis, we refer to  the classical books \cite{Am} or \cite{DGS}.

 \medskip
 
Concerning the field $K$, we assume as above that it is of characteristic zero, algebraically closed and complete with respect to a nontrivial absolute value $|\cdot |$.

The \emph{value group on $K$} is the set $|K^*|=\{|x|:x\in K^*\}\subset \R_{> 0}$; $K$ being algebraically closed implies that $|K^*|$ is a dense subset of $\R_{> 0}$.

When the absolute value is ultrametric we usually denote by $\O=\{x\in K:|x|\le 1\}$ the valuation ring and by $\mathfrak{M}=\{x\in K:|x|<1\}$ the maximal ideal.
We shall use the standard notations

\begin{align*}
 D(a,T^-)&:=\{x\in K:|x-a|<T\}, &  D(a,T^+)&:=\{x\in K: |x-a|\le T\}.
\end{align*}

\medskip
 
We will also refer to the \emph{Gauss norm} on the field of rational functions $K(X)$. This can be characterized as the unique extension of $|\cdot |$ to $K(X)$ such that the reduction of $X$ is transcendental over the residue field of $K$. For two polynomials $P(X)=\sum_i P_i X^i$ and $Q(X)=\sum_i Q_i X^i$ in $K[X]$, this can be concretely described as
$$\abs{P(X)/Q(X)}_G=\underset{i}{\max}\{|P_i|\} / \underset{i}{\max}\{|Q_i|\}.$$

 \medskip
 
Concerning the polynomial $f(X,Y)\in K[X,Y]$, we assume it is irreducible, and we denote by $d>0$ its degree in $Y$, which is then the degree $[K(x,s(x)):K(x)]$ of $s(x)$ as an algebraic element over the field $K(x)$.   

\medskip

The equation $f=0$ defines an irreducible plane algebraic curve ${\mathcal Z}\subset \A^2$, and we let $\widetilde{\mathcal Z}$ denote a smooth complete model of $\mathcal Z$, so that the function field of $\widetilde{\mathcal Z}$ over $K$ is $K(\widetilde{\mathcal Z})=K({\mathcal Z})\cong_K K(x,y)$, where $x$ is transcendental over $K$ and $f(x,y)=0$.

Also, $K(x,y)$  is embedded (over $K$)  through $y\mapsto s(x)$ into the field  $K((x))$ of formal power series, and it is isomorphic over $K(x)$ to $K(x,s(x))\subset K((x))$.
Note that the degree $\deg(x)$ of $x$ viewed as a rational function on $\widetilde{\mathcal Z}$ (namely the weighted number of its poles)  is $d=[K(\widetilde{\mathcal Z}):K(x)]$.\footnote{In particular, the integer $d$, which is the {\it degree of $s(x)$ as an algebraic element over $K(x)$},  must not be confused with {\it the degree of $s(x)$ as a function on $\widetilde{\mathcal Z}$}, which equals the degree of $f$ with respect to $X$ (not $Y$).
}

After multiplying  $s(x)$ by the leading coefficient of $f(X,Y)$ as a polynomial in $Y$, we may assume at the outset  that $f$ is monic in $Y$ (we have removed the `polar' singularities of $s(x)$).   This amounts to suppose that  $s(x)$  is integral over $K[x]$.

With such normalizations, if $d=1$ then $s(x)$ is a polynomial, and vice versa, so we will assume in the sequel that $d>1$.  In this case the function $x$ on $\widetilde {\mathcal Z}$, being of degree $d>1$,  will have some (finite) branch points, i.e. critical values.\footnote{If $x_0\in K$ is a  branch point then $f(x_0,Y)$ has some multiple root, but this is not a sufficient condition; consider e.g. the simple example $f(X,Y)=Y^2-X(X-x_0)^2$, $x_0\neq 0$ (what fails here is the smoothness of the model).}  This  existence is well known and follows easily from the Riemann-Hurwitz formula;  in the complex case it is a consequence of the fact that $\C$ is simply connected; it  will also follow implicitly from some of the arguments below.

\medskip

Now, if $x_0\in K$, let $p_1,\dotsc ,p_r$ be the points of $\widetilde{\mathcal Z}$ {\it above $x_0$}, i.e. with with $x(p_i)=x_0$. Then there are integers $e_1,\dotsc ,e_r\ge 1$ (the {\it ramification indices} of the $p_i$) such that the solutions of the equation $f(x,Y)=0$ may be expressed as {\it Puiseux series} $y=s_i(u)\in K[[u]]$, where $u^{e_i}=x-x_0$, convergent for $u$ in a neighborhood of $0$. We have $e_1+\dotsb +e_r=d$. Similarly if we work on $\P_1(K)$ and $x_0=\infty$; in this case the solutions may be expressed as Puiseux series {\it at infinity} $y=s_i(u)$ where $u^{e_i}=1/x$, the series converging for large enough $|u|$. We remark that
the Puiseux series at infinity in fact are {\it Laurent series}, i.e. they may contain (finitely many)  terms with {\it negative} exponent of $u$; this cannot happen at finite points due to the assumption that $f$ is monic in $Y$.

\medskip

 By definition, the value $x_0$ is not a branch point if and only if $r=d$, so $e_i=1$ for each index $i$. Then there are $d$ distinct points $p_1,\dotsc ,p_d$ of $\tilde X$ with $x(p_i)=x_0$, and the equation $f(x,Y)=0$ will have $d$ distinct power series solutions $y=s_i(x-x_0)\in K[[x-x_0]]$ convergent for $x$  in some  neighborhood of $x_0$.  If $x_0$ is a branch point and $e_i>1$, then the point $p_i\in\widetilde{\mathcal Z}$ is said to be {\it ramified} above $x_0$.

 \medskip
 
  In the complex case this convergence provides a process of analytic continuation of the original series $s(x)$, which will go on throughout  any simply connected region, as soon as we do not find branch points in it. 
 As a consequence it may be deduced (see below for explicit arguments) that  {\it for $K=\C$ the open disk of convergence of $s(x)$ cannot  contain all the (finite) branch points $x_0\in\C$}. This yields a negative answer to the above Question, proving the case \ref{T.Y:i} of Theorem \ref{T.Y} (of course this goes back to centuries).
 
 On the other hand, if the absolute value of $K$ is ultrametric (for instance if $K$ is the completion $\C_p$ of an algebraic closure of $\Q_p$) the analytic continuation is not anymore possible, hence the argument sketched for $\C$ fails.
 
 \medskip

\section{Proof of Theorem~\ref{T.Y}}

In this section we work under the normalizations performed in \S~\ref{S.notation}, in particular assuming that $f(X,Y)$ is monic in $Y$. This does not affect the statement.

\subsection{Proof of Theorem~\ref{T.Y}\ref{T.Y:i}}
 We start with part \ref{T.Y:i}, namely proving that the Question has a negative answer when $K=\C$.  As said, this case  is somewhat implicit in the standard theories, but lacking an explicit reference we are inserting a complete argument.

\medskip

We assume that the series $s(x)$ is convergent in a disk $D(0,R^-)$ containing all the finite branch points of $x$, and have to prove that $s(x)$ is a polynomial. Since there are only finitely many   branch points, our assumption yields that   any finite branch point $\rho$ is such that $|\rho|<r<R$ for a suitable $r>0$.

For a real number $R_1\ge R$ consider the open region $A=A(R_1):=\{z\in \C\setminus\R^+: r<|z|<R_1\}$. This is clearly simply connected and does not contain any branch point.  The convergence assumption and the fact that the polynomial $f$ is monic in $Y$  entail that  the function $s(x)$ is analytic and well defined in the annulus   $\Omega: r<|z|<R$, and  is bounded there. Also, $A\cap\Omega$ is connected, and therefore  $s(x)$ may be analytically continued throughout  $A$, hence throughout the union of $A$ with $\Omega$.   By the same reason, it can be analytically continued in the  union of $\Omega$  with the open (simply connected) rectangle  $\{z\in\C: r<\Re z<R_1, |\Im z|<\epsilon \}$, for  $\epsilon >0$.

The analytic continuations obtained through these two connected regions will  coincide on their intersection, which is connected and contains a  nonempty open  set (e.g. a suitable neighborhood of the interval $(r,R)\subset \R$), thus this  will provide an analytic continuation to the union of the regions, which contains the connected annulus $r<|z|<R_1$. Again, since the function $s(x)$ is analytic in the disk $|z|<R$, the  continuation that we have obtained extends to the disk $|z|<R_1$. Since $R_1>R$ is arbitrarily large, this yields that $s(x)$ may be continued to an analytic function on the whole $\C$.
 
But since it satisfies  the monic equation  $f(x,s(x))=0$, it is subject to a bound $|s(x)|\le c_1|x|^l+c_2$, provided $c_1,c_2,l$ are sufficiently large positive constants.  By Cauchy's integral formula  $a_m=(2\pi i)^{-1}\int_{|z|=T}s(z)z^{-m-1}\d z$ (take any $T>1$);  this implies that $|a_m|\le (c_1+c_2)T^{l-m}$ for all $T>1$ and hence $a_m=0$ for $m>l$,\footnote{Of course this argument is Liouville's theorem.} so that $s(x)$ is a polynomial of degree $\le l$.\qed

\begin{rem}
There are other equivalent ways to formulate the argument with analytic continuation.

For instance one can observe at once that the   continuation to $A$ indeed extends to the whole ring $r<|z|<R_1$. In fact, travelling through a circle  $|z|= c$, (where $r<c<R_1$) from $c$ will lead to the same value for $s(z)$ after one complete round, because this holds for $c<R$ by assumption, and must then hold always by continuity.

A further somewhat different argument can be obtained as in the proof that we shall offer for  the $p$-adic case; this argument in the complex case may be sketched as follows.  Replacing $x$ with a positive power $x^e$ does not affect the issue, so we may assume that $\infty$ is not ramified. Then we may analytically continue the restriction of  $s(z)$ to  $ r<|z|<R$, to the whole simply connected domain $r<|z|\le\infty$ in $\P_1(\C)$.  This continuation must fit with the function  $s(z)$ defined by the series for $|z|<R$, because the intersection of the domains $|z|<R$ and $r<|z|$ is connected.

Finally, note that any of these arguments proves as a byproduct that in the complex case there must exist finite branch points as soon as $d\ge 2$.
\end{rem}

\subsection{Proof of Theorem~\ref{T.Y}\ref{T.Y:ii}}
We now go ahead with part~\ref{T.Y:ii}, which is one of the main points of the present paper. As mentioned in the introduction, we apply a theorem of Dwork and Robba, and precisely Theorem 2.1 of \cite{DR}.
For the reader's convenience we state here the modification that we need, which was already stated, proved,  and used, in the paper \cite{Z3} by the second author: see Proposition~4.1 therein.\footnote{There are subsequent more sophisticated results of Dwork, Robba, and others, which could be invoked but, to keep the pre-requisites to a minimum, we prefer to use this elementary version.}

For the statement we keep the previous notations and further denote by $u$ a variable over $K$, and by $\widehat{K(u)}$ the completion of $K(u)$ with respect to the Gauss norm on $K(u)$.

\begin{prop}[\cite{DR}, Thm. 2.1 and \cite{Z3}, Prop. 4.1.]\label{P.DR}    Let $Q(Z)=Q(u,Z)\in K[u,Z]$ be irreducible of degree $d>0$ in $Z$, and suppose that at each $u_1\in D(0,1^-)$, except possibly at $u_1=0$, the equation $Q(Z)=0$ has $d$ distinct locally analytic solutions.

Factor $Q$ over $\widehat{K(u)}$ and consider the finite extensions of $\widehat{K(u)}$ corresponding to the various irreducible factors. Assume that each such extension is tamely ramified and that the corresponding residue field extension  is separable.

Then the solutions of  $Q(Z)=0$ at $u=0$ are of the form
\begin{equation}\label{E.DR}
Z=u^{-m/e}\xi(u^{1/e}),
\end{equation}
for an integer $m\ge 0$, where $e$ is some ramification index above $u=0$ of the function field extension of $K(u)$ defined by $Q$, and where $\xi\in K[[U]]$ is a power series whose radius of convergence is at least $1$.
\end{prop}

\medskip

\begin{rem}
Theorem~2.1 of \cite{DR} has the same conclusion but contains merely the assumption that ($p=0$ or) $p>d$, and does not even mention the factorization of $Q$ over   $\widehat{K(u)}$. In fact, the possible factors will have degree $\le d$, so the constraints about separability and tame ramification become automatic on assuming $p>d$. However a little inspection of  the arguments in \cite{DR} readily  leads to the present refinement.  That some condition on $p$ is needed is observed in \cite{DR}, where the simple example $Q(u,Z)=Z^p-(1+u)$ appears; in that case we have $p=d$, the assumption about the local solutions is fulfilled, but the radius of convergence is $<p^{-1}$ (in fact, it is $p^{-p/(p-1)}$). Our situation is different and the (counter)examples that we give offer implicitly different (counter)examples concerning the theorem of Dwork and Robba.
\end{rem}

\medskip

Before applying Proposition~\ref{P.DR}, we perform some normalizations, and for simplicity we  first discuss the case $p>d$.

\medskip

As before,  we assume that the series $s(x)$ is integral over $K[x]$,  that it is convergent in a disk $D(0,R^-)$, and that all the finite branch points of $x$ have absolute value $<R$; we have to prove that $s(x)$ is a polynomial. 

\medskip

By a rescaling $x\mapsto \mu x$, where $\mu\in K^*$ is such that $\abs{\mu}$ is a bit smaller than the radius of convergence, we can suppose that $s(x)$ converges in $D(0,r^-)$ for an $r>1$, and that the branch points have absolute value $<1$: indeed the rescaling divides  the radius and the branch points by $\mu$.  Also, by changing $x$ with a suitable positive power $x^l$, we can suppose that
$\infty$ is not a branch point: it suffices to choose $l$ as a positive multiple of all the ramification indices at $\infty$ (so e.g. $l=d!$ will do). Note that this substitution will possibly increase the number of branch points\footnote{The `new' branch points will be the $l$-th roots of the `old' ones.} but their absolute values will remain $<1$.  After this normalization, let then $\rho <1$ be an upper bound for the absolute values of the branch points.

\medskip

Put $u=1/x$ and consider the polynomial $Q(u,Z)=u^hf(1/u,Z)$, where $h$ is the minimal integer so that $Q\in K[u,Z]$ is a polynomial, which will then be irreducible.  Now, the finite branch points $u_0$ of $u$ are of the shape $u_0=1/x_0$, where $x_0$ is a nonzero branch point of $x$. Hence $|u_0|\ge 1/\rho>1$. Then, since the disk $D_u(0,1^-):|u|<1$ has no branch points, the solutions of $Q(u,Z)=0$ around any $u_1$  in that disk will be given locally by $d$ distinct power series in $u-u_1$, with nonzero radius of  convergence. 

Therefore, since we are assuming that $p>d$, the other assumptions for Proposition \ref{P.DR} are automatically satisfied; the conclusion is that the solutions of $Q(u,Z)$ at $u=0$ are given by Laurent series of the shape $u^{-m_i/e_i}\xi_i(u^{1/e_i})$, where the $\xi_i$ are power series with radius of convergence at least $1$. Now, since $\infty$ is a not a branch point of $x$, we deduce that $0$ is not a branch point of $u$, whence we may take each $e_i=1$, and we shall have $d$ distinct solutions of the said shape, namely of the form $\theta_i(x^{-1}):=u^{-m_i}\xi_i(u)=x^{m_i}\xi_i(x^{-1})$.   

These last series are Laurent series in $1/x$ with only finitely many terms with positive exponent of $x$.  By Proposition \ref{P.DR}, they are convergent for $0<|u|<1$, i.e. for $|x|>1$.  In practice, the proposition implies that the Puiseux series at $x=\infty$ are convergent for $|x|>1$.

Now, in the annulus $1<|x|<r$, the polynomial $f(x,Z)$ factors
as $\prod_{i=1}^d(Z-\theta_i(x^{-1}))$.  But, in the same annulus,  we may evaluate this at $Z=s(x)$  obtaining $\prod_{i=1}^d(s(x)-\theta_i(x^{-1}))=0$. But the Laurent series convergent in the annulus form a domain, and then we deduce that there exists an index $j$ such that the equality $s(x)=\theta_j(x^{-1})$  holds for all $x$ in the annulus, and in fact, equivalently,  holds identically.  However $s(x)$ is a power series, whereas $\theta_j(x^{-1})$ has only finitely many terms with positive exponent of $x$, and therefore $s(x)$ is a polynomial.

\begin{rem}
This argument works for $K=\C$ as well, a case for which  the Proposition is in fact part of standard knowledge, and does not need assumptions concerning tame ramification or separability.

For a similar reason, it also works for ultrametric $K$ with residue field of characteristic $0$, although the (subtler) arguments of Dwork and Robba \cite{DR} are needed with respect to the case of $\C$. In any case, this leaves us only with the case of residue field of characteristic $p=d$, which we are now going to consider.
\end{rem}

Now suppose that $p=d$. We start with the same normalization as above, hence supposing that $s(x)=\sum_{m=0}^\infty a_mx^m$ converges in $D(0,r^-)$ where $r>1$, and supposing that every branch point has absolute value $\le \rho<1$.

As in the previous argument, we set $u:=1/x$ and we consider  the polynomial $Q(u,Z)=u^hf(1/u,Z)$, defined above. Now, since $s(x)$ has radius of convergence $>1$, the coefficients $a_m$ tend to $0$. Hence  $s(x)$ may be considered an element of  the completion $\widehat{K(u)}=\widehat{K(x)}$ of $K(u)=K(x)$ with respect to the Gauss norm; in fact $|s(x)-\sum_{m=0}^na_nu^{-n}|_G\le \sup_{l>n}|a_l|\to 0$ as $n\to\infty$, so $s(x)$ is a limit (in the Gauss norm) of a sequence in $K(u)$. 

 Therefore $Q(u,Z)$ has the linear factor $Z-s(x)$ over $\widehat{K(u)}$, the quotient being a polynomial of degree $d-1<p$  in $Z$. We conclude that the extensions mentioned in the statement of Proposition \ref{P.DR} have degree $\le d-1<p$ and therefore are neither wildly ramified nor the residue field extension can be inseparable.

But then we may apply Proposition~\ref{P.DR} as in the case $p>d$ and conclude the proof in the same way as above.\qed

\subsection{Proof of Theorem~\ref{T.Y}\ref{T.Y:iii}}
Now suppose that there exists $x_0\in K$ such that all the ramification indices above it are $>1$; this implies $d>1$, and we have to prove that it is not possible that $s(x)$ converges in a disk $D(0,R^-)$ containing all the finite  branch points.

This is indeed very easy. The assumption entails  that  $|x_0|<R$, and then we could write $s(x)=\sum_{m=0}^\infty \frac{s^{(m)}(x_0)}{m!}(x-x_0)^m$, the series being again convergent for $|x-x_0|<R$.  But this expansion yields a Puiseux series   at $x=x_0$ with ramification index $1$, contrary to assumptions.\qed

\subsection{Proof of Corollary~\ref{C.Yii}}
 Let $R$ be the radius of convergence of $s(x)$. In view of the assumptions of the Corollary, we may apply Theorem~\ref{T.Y}, so there is a finite branch point $\rho$  with $R\le |\rho|$ (in particular $R$ must be finite).\footnote{The finiteness of $R$ does not depend on Theorem~\ref{T.Y}: for otherwise $s(x)$  would be necessarily a polynomial since it is analytic on the whole $K$ and is bounded polynomially for $|x|\to\infty$. The conclusion then follows as in Liouville  classical result over $\C$; see for this e.g. \cite{Am} or \cite{DR}.}

 Let then $\rho$ be such a branch point of minimum absolute value. If $|\rho|=R$ we are done, and similarly if there exists a branch point $\beta$ with $|\beta|\le  R$. 
 So assume $|\rho|>R$ and that there does not exist any such $\beta$, and  put $s_1(x)=s(\rho x)$. Then the equation $f(\rho u,Z)=0$ has $d$ distinct locally analytic solutions for $u\in D(0,1^-)$. 
 But then, in view of the present assumptions, we may apply Proposition \ref{P.DR} to the polynomial $Q(u,Z):=f(\rho u,Z)$.\footnote{Observe that in case  $p=d$ the same argument as in the proof of Theorem \ref{T.Y}\ref{T.Y:ii} would not work, since now  we lack the factor $Z-s(x)$ of $f(x,Z)$ over $\widehat{K(x)}$.} We conclude that the radius of convergence of $s_1$ is at least $1$, which implies that the radius of convergence of $s(x)$ is at least $|\rho|$, a contradiction which proves the statement.\qed
 
 \begin{rem}
We remark again that in the case of $\C$ the proof of Theorem~\ref{T.Y} gives the more precise conclusion that $R=|\rho|$ for some branch point $\rho$, which in the ultrametric case is not in general true.
 \end{rem}

\section{Proof of Theorem~\ref{T.N}}

We begin with part \ref{T.N:i} and with an example for $p=2$ that admits a very explicit description.

\subsection{An example for \texorpdfstring{$p=2$}{p=2}}
Consider the series satisfying the algebraic equation
\[
 y^3-3y=2-x
\]
with $y(0)=2$.
The branch points are at $x=0$ (above which there are a ramification point with $y=-1$ and an unramified point with $y=2$) and at $x=4$ (above which there is the ramification point with $y=1$ and an unramified point with $y=-2$).

We show in two ways that this series belongs to $\Z_2[[x]]$. In particular this implies that it converges for $\abs{x}_2 < 1$ which is an open disk that contains both finite branch points, thus proving the statement.

\subsubsection{Lift via Hensel's lemma}
We will construct the series  $\sum s_nx^n$ iteratively, starting from $s_0=2$, which is a simple root of the polynomial $f(Y)=Y^3-3Y+x-2\in \Q_2[[x]][Y]$ modulo $(x)$, meaning that $f(2)=x$ belong to the maximal ideal of $\Q[[x]]$, while $f'(2)=9$ does not.
By Hensel's lemma this root lifts iteratively to a unique root of $f$ in $\Q_2[[x]]$ which is precisely the series we are interested in. This lifting is achieved by the following recursion, which as usual can be interpreted as an ultrametric version of Newton's method for approximating roots numerically:
\begin{equation*}
s_0=2,\qquad s_{n+1}=s_n-\frac{f(s_n)}{f'(s_n)},\quad n\ge 0.
 \end{equation*}
We see by induction that the elements $s_n$, which in principle belong to $\Q_2[[x]]$, are actually in $\Z_2[[x]]$. Indeed we have that $f'(s_n)=3(s_n^2-1)=9+O(x)$ because $s_n\equiv s_0 \pmod{(x)}$; assuming by induction  that $s_n\in\Z_2[[x]]$ this implies that $f'(s_n)$ is an invertible element of $\Z_2[[x]]$, and therefore $s_{n+1}\in\Z_2[[x]]$.
The coefficients of $s_n$ stabilize as $n\to\infty$ and therefore the limit belongs to $\Z_2[[x]]$ as desired.

\subsubsection{An explicit formula involving Chebyshev's polynomials}
Through some manipulations, it is actually possible to obtain a fairly explicit description of the series expressing the desired algebraic function and check directly that its coefficients are $2$-integers. We include this because we hope the readers will find it amusing.

Solving for $y$ using Cardano's formula gives
\begin{align*}
 y&=\left( \frac{2-x}{2}+\sqrt{\frac{(2-x)^2}{4}-1}\right)^{1/3}+\left( \frac{2-x}{2}-\sqrt{\frac{(2-x)^2}{4}-1}\right)^{1/3}\\
 &=\left( 1-\frac{x}{2}+\sqrt{x\left(\frac{x}{4}-1\right)}\right)^{1/3}+\left(1- \frac{x}{2}-\sqrt{x\left(\frac{x}{4}-1\right)}\right)^{1/3},
 \end{align*}
 Where the square roots are taken with opposite signs in the two factors, and the cubic roots are meant to have  value 1 at  $x=0$.
 Expanding the cubic roots and writing $u^2=x$ we get
 \begin{align}
 y&=\sum_{n=0}^\infty \binom{1/3}{n}\left( \left(-\frac{x}{2}+\sqrt{x\left(\frac{x}{4}-1\right)}\right)^n +  \left(-\frac{x}{2}-\sqrt{x\left(\frac{x}{4}-1\right)}\right)^n              \right)\\
 &=\sum_{n=0}^\infty \binom{1/3}{n} u^n \left( \left(-\frac{u}{2}+\sqrt{\frac{u^2}{4}-1}\right)^n +  \left(-\frac{u}{2}-\sqrt{\frac{u^2}{4}-1}\right)^n              \right)\\
 \label{counterex:p2:Cheby}&=\sum_{n=0}^\infty 2\binom{1/3}{n} u^n  T_n\left(-\frac{u}{2}\right)
\end{align}
where $T_n$ is the $n$-th Chebyshev polynomial, satisfying $T_n(\cos\theta)=\cos(n\theta)$.
This can be seen for example from the identity $T_n\left(\frac{X+X^{-1}}{2}\right)=\frac{X^n+X^{-n}}{2}$, noticing that 
$$ 
\left(-\frac{u}{2}+\sqrt{\frac{u^2}{4}-1}\right)\left(-\frac{u}{2}-\sqrt{\frac{u^2}{4}-1}\right)     =1.
$$
Notice that $T_n$ is an odd polynomial for odd $n$ and an even polynomial for even $n$, thus $u^n  T_n\left(-\frac{u}{2}\right)$ is indeed a polynomial in $x$.
Furthermore, explicit formulae such as
\[
 T_n(X)=\frac{1}{2}\sum_{m=0}^{\floor{n/2}}(-1)^m\left(\binom{n-m}{m}+\binom{n-m-1}{n-2m}\right)(2X)^{n-2m}
\]
allow us to see that the coefficients of the series \eqref{counterex:p2:Cheby} are 2-integers as desired.

\medskip

\subsection{An example for all \texorpdfstring{$p$}{p}}
Consider the series $y=y(x)$ satisfying the algebraic equation
\begin{equation}
 y^p(y+1)=x
\end{equation}
and $y(0)=-1$. The finite branch points of the equation are at $x=0$ and $x=\frac{-p^p}{(p+1)^{p+1}}$.

Exactly as in the previous case, Hensel's lemma can be used to show that the series $y(x)$ is in $\Z_p[[x]]$,
thus showing that all finite branch points are inside the disk of convergence.

The details are as follows: starting from the simple root $s_0=-1$ we study the recurrence
\[
 s_{n+1}=s_n-\frac{f(s_n)}{f'(s_n)}.
\]
As before, $f'(s_0)=(-1)^p$ is an invertible element of $\Z_p$, and therefore, by induction, the series $f'(s_n)$ is an invertible element of $\Z_p[[x]]$; thus the whole sequence $(s_n)$ lie in $\Z_p[[x]]$ and so does its limit.

\begin{rem}
 It is easy to see that one can recover from these functions other examples that show how the condition $p>d$ cannot be removed from the theorem of Dwork-Robba. Indeed if the series for $y(x)$ converges strictly further than all branch points, than upon moving the point at infinity to the origin the radius of convergence will be strictly smaller than the distance to the first branch point. In formulae, setting $x=1/t^{p+1}, z=ty$ gives the algebraic equation $z^p(z+t)=1$ and the series $z(t)$ provides the example. These examples are of a different shape than those proposed in \cite{DR}.
\end{rem}

\subsubsection{Fuss-Catalan numbers and an explicit formula}
By differentiating the equation $y(x)^p(y(x)+1)=x$ with respect to $x$ we find $y' y^p ((p+1)y+p)=1$, and thus, using the original equation, we see that
\[
 y' ((p+1)y+p)=y\frac{y+1}{x}.
\]
By writing $y=\sum_{n\geq 0}a_n x^n$ we can identify the coefficient of $x^n$ in the left-hand side of the equation as
\[
 (p+1)\sum_{i=1}^n i a_i a_{n-i+1}-(n+1)a_{n+1},
\]
and the coefficient of $x^n$ in the right-hand side of the equation as
\[
 \sum_{i=0}^n a_i a_{n-i+1},
\]
so that we obtain the recurrence
\[
 a_{n+1}=\frac{1}{n}\sum_{i=1}^n a_i a_{n-i+1}\left(i(p+1)-1\right)\qquad n\geq 1
\]
and $a_0=a_1=-1$.
Up to the sign, these are integers known in combinatorics as the Fuss-Catalan numbers and they appear in the enumeration of certain trees or certain lattice paths as a generalization of Catalan numbers $\frac{1}{n+1}\binom{2n}{n}$. (See Stanley's volumes \cite{S} for many similar examples and much more.) They have an `explicit' formula as
\[
 a_n=\frac{(-1)^n}{(p+1)n-1}\binom{(p+1)n-1}{n}=\frac{(-1)^n}{pn-1}\binom{(p+1)n-2}{n}\qquad n\geq 1.
\]

In all of these examples the origin is itself a branch point, though of course the ramification point lies on a different branch. (We found that putting the branch point at $0$   was helpful for the computations involved in finding these series.)

\subsection{Proof of Part~\ref{T.N:ii}} The simplest examples here are perhaps obtained  (as in \cite{DR}) by taking $f(X,Y)=Y^p-1-X$, over the field $\C_p$: now $s(x)^p=1+x$, so $s(x)=\zeta\sum_{m=0}^\infty \binom{1/p}{m}x^m$, where $\zeta^p=1$. All such series have radius of convergence $p^{-p/(p-1)}$, whereas the only finite branch point is $-1$, with ramification index $p$.

A less evident example, with ramification index $2$, comes from the polynomial $f(X,Y)=Y^{p-1}(Y-1)-X$. It is irreducible of degree $p$ in $Y$, and in fact the point $x=\infty$ is totally ramified. The finite branch points for $x$ are obtained from the equation $f_y=(p-1)y^{p-2}(y-1)+y^{p-1}=0$, so they are $0$ (for $p>2$, with ramification indices $p-1,1$)  and $\rho:=-(p-1)^{p-1}p^{-p}$ (obtained with $y=(p-1)/p$, with ramification index $2$).

At $x=0$ we have a unique series solution $s(x)=\sum_{m=0}^\infty a_mx^m\in \C_p[[x]]$ (in fact over $\Q$) with  $a_0=1$ and $f(x,s(x))=0$.
As we have seen before, the fact that $f_y(a_0)=1$ is invertible in $\Z$ implies that $s(x)\in\Z[[x]]$.

Assuming now that $s(x)$ has radius of convergence $>1$,   it  would reduce modulo $p$  to a polynomial  $t(x)$,
but the congruence $f(x,t(x))\equiv 0\pmod p$ has clearly no polynomial solutions. Hence $s(x)$ has radius of convergence $1$, which is smaller than $|\rho|_p=p^p$.\qed

\medskip

This argument shows that for a fixed $f$ the radius of convergence is `usually' $1$, and it is  $>1$ when the irreducible polynomial $f$ acquires a factor $Y-P(X)$ when seen modulo $p$.

\medskip

\section{Proof of Theorem~\ref{T.BDR}}

  We  put
  \begin{equation*}
  s(x)=\sum_{m=0}^\infty a_mx^m,\qquad  s_n(x)=\sum_{m=0}^{n-1}a_mx^m,\quad \rho_n(x):=s(x)-s_n(x).
  \end{equation*}

Since $r$ is in the value group, up to a rescaling of $x$ we may take  $r=1$;  our assumption now corresponds to the fact that $v(a_m)\to \infty$ as $m\to\infty$. By multiplying $s(x)$ by a nonzero  constant we can also assume that $v(a_m)\ge 0$ for every $m\ge 0$, and that $f(x,Y)$ is the  minimal polynomial of $s(x)$  over $\O[x]$.

\medskip

\begin{rem}\label{T.BDR:NOTE}
As a side observation, we remark that the result is easier for rational functions, in which case we may argue as follows. The assertion is trivial for polynomials, so by induction  it suffices to prove that  if $a\in K^*$ and if $S(x):=(x-a)s(x)$   has radius of convergence $R>1$, while $s(x)$ converges for $|x|\le 1$, then $s(x)$ as well has radius of convergence $>1$.  To prove this, note that  if $|a|>1$ the sought conclusion  is clear, because $(x-a)^{-1}=-\sum_{n=0}^\infty a^{-n-1}x^n$
has radius of convergence $|a|$, so assume $|a|\le 1$. Then we can evaluate at $x=a$ both sides, to find $S(a)=0$. Hence, setting $S(x)=\sum_{n=0}^\infty A_nx^n$,   $$S(x)=S(x)-S(a)=(x-a)\sum_{n=0}^\infty A_n(x^{n-1}+ax^{n-2}+\dotsb +a^{n-1}),$$ thus $s(x)=\sum_{n=0}^\infty A_n(x^{n-1}+ax^{n-2}+\dotsb +a^{n-1})$, and  the conclusion follows very easily.\footnote{This argument, and in fact  the case of rational functions, remain valid also over $\C$.}

Note also  that
 this same argument would allow us to suppose that $f(X,Y)$ is monic in $Y$, though we do not need that.
\end{rem}

 In light of the remark, we shall suppose from now on $d\ge 2$.
 Going ahead with the general case, we let $|\cdot |_G$ denote the Gauss norm\footnote{For a series as above we define it as $\sup_n|a_n|$; see \cite{DGS}, Ch. IV,  for more.} and $v_G=-\log|\cdot|_G$ the corresponding order function, and set $h:=v_G(f_Y(x,s(x)))$, for the derivative $f_Y$ with respect to $Y$; since $f(X,Y)$ is irreducible we certainly have $f_Y(x,s(x))\neq 0$, so $h$ is well-defined, and it is $\ge 0$.

  Clearly  $f_Y(x,s_n(x))\xrightarrow[G]{} f_Y(x,s(x))$ for $n\to\infty$, hence $v_G(f_Y(x,s_n(x)))$  is constantly equal to $h$ for all large $n$, in fact it suffices that $v(a_m)>h$ for $m\ge n$.

   \medskip

  Let now $l$ be a fixed integer large enough so that for all $n\ge l$ we have $v(a_n)>h$ and $v_G(f_Y(x,s_n(x))=h$ . Also  pick $x_0\in K$ with $|x_0|\le 1$ and such that $|f_Y(x_0,s_l(x_0))|$ is maximum for $x$  in the closed unit disk. It is a well-known easy fact that such an $x_0$ exists, and the maximum is the Gauss norm of  the polynomial $f_Y(x,s_l(x))$.
  Note that $|f_Y(x_0,s_l(x_0))|=|f_Y(x,s_l(x))|_G=|f_Y(x,s(x))|_G$ and $v(f_Y(x_0,s_n(x_0)))=h$ for all $n\ge l$ since for such $n$ we have $v(a_n)>h$.

Now, the assumptions of convergence remain valid for $s(x_0+x)$ in place of $s(x)$. In fact  we have the following identities, justified by absolute convergence:

$s(x_0+x)=\sum_{m=0}^\infty a_m(x+x_0)^m=\sum_{r=0}^\infty (\sum_{m=r}^\infty a_m \binom{m}{r} x_0^{m-r})x^r$,
which  converges for $|x|\le 1$ because $a_m\to 0$ and $|\sum_{m=r}^\infty a_m \binom{m}{r}x_0^{m-r}|\le \underset{{m\ge r}}{\sup}|a_m|$.

Hence we may perform a translation and assume $x_0=0$.   With this normalization, observing  that $s(0)=a_0$, we have $0\le v(f_Y(0,a_0))=h= \underset{|x|\le 1}{\inf}v(f(x,s_l(x)))$.

  \medskip

  We want to perform another useful normalization.
  We put $$F(x,Y)=F_l(x,Y):=f(x,s_l(x)+Y)=c_d(x)Y^d+ c_{d-1}(x)Y^{d-1}+\dotsb + c_0(x),$$ so for instance $ c_0(x)=  F(x,0)=f(x,s_l(x))$. We have that $v_G(c_0(x))\to 0$ for $l$ tending to infinity, since $f(x,s_l(x))\xrightarrow[G]{} f(x,s(x))=0$.

  \medskip

  We have $c_1(x)=F_Y(x,0)=f_Y(x,s_l(x))$, whence $c_1(0)=F_Y(0,0)=f_Y(0,a_0)$.

  Hence  $v(c_1(0))=v_G(f_Y(x,s_l(x)))=h=v_G(c_1(x))$.

  \medskip

  We also have $F(x,\rho_l(x))=F(x,s(x)-s_l(x))=f(x,s(x))=0$.

  \medskip

 For notational convenience put now for this argument $b:=c_1(0)$ (we have $v(b)=h$, so $b\neq 0$) and   write $\rho_n(x)=:b\rho_n^*(x)$. Since $\rho_n(x)=s(x)-s_n(x)\xrightarrow[G]{} 0$,  we have that $v_G(\rho_n^*)\to\infty$ as $n$ grows to infinity.

  We   also put
  $$
  F^*(x,Y):=b^{-2}F(x,bY)=b^{d-2}c_d(x)Y^d+\dotsb +c_1(x)b^{-1}Y+ c_0(x)b^{-2}.
  $$

   \medskip

  Now observe that $F^*(x,\rho_l^*(x))=0$, and $F^*_Y(0,0)=c_1(0)b^{-1}=1$.

  Hence, on fixing in advance a large enough $l$, and changing $f$ with $F^*$,    replacing $s(x)$ with $\rho^*_l(x)$, $s_n(x)$ with $\rho^*_l(x)-\rho_{n+l}^*(x)$, we can assume at the outset that $v(a_n)\ge h_0>0$, $n=0,1,\dotsc$, for an arbitrarily  chosen $h_0$, as large as we want: it suffices to choose $l$ so that $v(a_m)\ge h_0+2v(b)=h_0+2h$ for $m\ge l$.

  Note also that with this normalization we have $a_m=0$ for $m<l$.

  \medskip

 In order to refresh the notation we  write down explicitly an expression for the new data:
  \begin{equation*}
  f(x,Y)=c_d(x)Y^d+c_{d-1}(x)Y^{d-1}+\dotsb +c_1(x)Y+c_0(x)\in \O[x,Y],
  \end{equation*}
  where $v_G(c_i)\ge 0$, $v(c_1(0))=0=v_G(c_1(x))$. We also note that  $f_Y(0,0)=c_1(0)=1$.

  \medskip

  We now recall a recurrence formula which is easily derived. Namely, we have
  \begin{equation*}\label{E.rec}
  a_n=-\hbox{coefficient of $x^n$ in} \quad f(x,s_n(x)).
  \end{equation*}

This formula appears  for instance in  the section on Puiseux series in \cite{Z2}, but it is also simple to prove it  directly. We have $s_{n+1}(x)=s_n(x)+a_nx^n$ and $f(x,s_{n+1}(x))=O(x^{n+1})$ (since $f(x,s(x))=0$). Also, $$f(x,s_n(x)+a_nx^n)=f(x,s_n(x))+a_nx^n f_Y(x,s_n(x))+O(x^{n+1})=f(x,s_n(x))+a_nx^n+O(x^{n+1}),$$ in view of the present normalization $f_Y(0,0)=1$. The assertion follows.

\medskip

To express the formula more explicitly, let  $D:=\underset{i}{\max}\deg c_i$, and write $c_i(x)=\sum_{j=0}^Dc_{ij}x^j$. For $n>\deg c_0$ the recursion yields,
\begin{equation}\label{E.rec2}
-a_n=\sum_{j=1}^Dc_{1j}a_{n-j}+\sum_{q=2}^d\left(\sum_{j,\kappa_1,\dotsc ,\kappa_q}c_{qj}a_{n-\kappa_1}\dotsm a_{n-\kappa_q}\right),
\end{equation}
where $q$ ranges from $2$ to $d$, whereas in the inner sums $j,\kappa_1,\dotsc ,\kappa_q$ satisfy
\begin{equation*}
 0\le j\le D,\quad 1\le \kappa_i\le n,\quad \sum_{i=1}^q(n-\kappa_i)=n-j.
\end{equation*}

Let $\eta=\min\{1,\min\{v(c_{1j}) : v(c_{1j})>0\}\}$ and $\varepsilon =\frac{\eta}{2D+1}$.
With the above normalizations, letting $h_0$ be large enough, we want to prove that
\begin{equation}\label{E.a_m}
v(a_m)\ge \varepsilon m+\eta,\qquad\hbox{for all $m$}.
\end{equation}

For $m\le D$  this is true because we are assuming that $v(a_m)\ge h_0$ for all $m$.

In general  we go ahead distinguishing two cases.

\medskip

{\bf First case}: {\it We have $v(c_{1j})>0$ for each $j\ge 1$}.

In this case  we argue by induction on $m$. Suppose the inequality true up to $m=n-1$.

Then  the term $c_{1j}a_{n-j}$  in the first sum on the right of \eqref{E.rec2}  has  order $\ge \underset{1\le j\le D}{\min} v(a_{n-j})+\eta$ which by induction is $\ge \varepsilon (n-D)+2\eta \ge  \varepsilon n+\eta$.

The terms in the summation corresponding to $q\ge 2$, by the inductive assumption, have order $\ge \varepsilon(\sum_{i=1}^q(n-\kappa_i))+q\eta \ge \varepsilon(n-D)+2\eta \ge \varepsilon n+\eta$.

Therefore we conclude that $v(a_n)\ge \varepsilon n+\eta$ and we are done by induction.

\medskip

{\bf Second case}: {\it There exists $j\ge 1$ with $v(c_{1j})=0$}.

In this case we let $\mu\ge 1$ be the largest such index $j$.

As $v(a_m)\to\infty$ each value $v(a_m)$ is attained only finitely many times, and therefore we may define an increasing sequence $0=\xi_0<\xi_1<\xi_2<\dotsb$, where $\xi_i$, $i>0$, are precisely the distinct values $v(a_m)$.

Now suppose the inequality $v(a_m)\ge \varepsilon m+\eta$ has been proved for all $m$ with $v(a_m)<\xi_e$, for a certain $e>0$. This holds for $e=1$ (empty case). Then  by induction we deduce it for all $m$ with $v(a_m)=\xi_e$.  This will prove the sought inequality for all $m$.

\medskip

For the inductive deduction we let $\mu>0$ be as above and, letting  $s$ be the maximum index for which $v(a_s)=\xi_e$,
 we define $n=\mu+s$.
Then we look again at the recurrence \eqref{E.rec2} for such $n$, and look at the various terms.

\medskip

{\tt Left-hand side}: We have $n>s$, hence $v(a_n)>\xi_e$.

{\tt Right-hand side}:  In a term $c_{1j}a_{n-j}$ with $j<\mu$ we have $n-j>n-\mu=s$, so $v(a_{n-j})>\xi_e$ by definition of $s$.

In a term $c_{1j}a_{n-j}$ with $j>\mu$, we have $v(c_{1j})>0$ hence either $v(a_{n-j})\ge \xi_e$, and then $v(c_{1j}a_{n-j})>\xi_e$, or $\xi_e>v(a_{n-j})\ge \varepsilon (n-j)+\eta$ by induction. In this last case we have $v(c_{1j}a_{n-j})\ge \varepsilon (n-j)+2\eta \ge \varepsilon n+\eta$.

In a term $c_{q,j}a_{n-\kappa_1}\dotsm a_{n-\kappa_q}$ in the second summation, hence with $q\ge 2$, either some index $n-\kappa_i$ is  such that a $v(a_{n-\kappa_i})\ge \xi_e$, and then $v(a_{n-\kappa_1}\dotsm a_{n-\kappa_q})>\xi_e$,  or $v(a_{n-\kappa_i})< \xi_e$ for all of the indices $n-\kappa_i$. In this case by induction we have $v(a_{n-\kappa_1}\dotsm a_{n-\kappa_q})\ge \varepsilon (\sum_i(n-\kappa_i))+q\eta \ge \varepsilon (n-j)+2\eta \ge \varepsilon n+\eta$.

\medskip

In conclusion, all the terms in the formula have order either $>\xi_e$ or $\ge \varepsilon n+\eta$, with the only possible exception of the term $c_{1\mu}a_{n-\mu}$. Since this last term has order $\xi_e$, it follows (from the ultrametric property) that $\xi_e\ge \varepsilon n+\eta$. Then if a term $a_{n-j}$ has order $\xi_e$  it follows that $v(a_{n-j})=\xi_e\ge \varepsilon n+\eta\ge \varepsilon (n-j)+\eta$, completing the induction step.

\medskip

Now from \eqref{E.a_m} the conclusion of the theorem follows at once.\qed

\section{Proof of Theorem~\ref{T.N2}}

In this section we will develop easy  lemmas to link the radius of convergence of the series $\sum_m \binom{\gamma_1}{m}\binom{\gamma_2}{m}x^m$ to the $2$-adic expansion of $\gamma_1,\gamma_2$. This will actually show that an uncountable number of pairs $\gamma_1,\gamma_2$ as in Theorem~\ref{T.N2} can be explicitely described.

\medskip

We recall that the $p$-adic valuation of a binomial coefficient can be computed in terms of the base-$p$ representation of the numbers.

A theorem of Kummer \cite{Kummer} states that for given integers $n \geq m \geq 0$ and a prime number $p$, the $p$-adic valuation $v_p\left(\binom{n}{m}\right)$ of the binomial coefficient $\binom{n}{m}$ is equal to the number of carries when $m$ is added to $n-m$ in base $p$.
This can also be expressed as
\[
 v_p\left(\binom{n}{m}\right)=\frac{S_p(m)+S_p(n-m)-S_p(n)}{p-1},
\]
where $S_p$ denotes the sum of the base-$p$ digits.

For $\gamma$ a $p$-adic integer the valuation $v_p\left(\binom{\gamma}{m}\right)$ can be expressed analogously in terms of the unique (infinite) base-$p$ expansion of $\gamma=\sum_{i=0}^\infty a_ip^i$ with integers $a_i\in [0,p-1]$. In particular the sum in base $p$ of $m$ and $\gamma-m$ involves only  finitely many carries unless the digits of $\gamma-m$ are eventually equal to $p-1$ and those of $\gamma$ to $0$; this happens when $\gamma$ is a positive integer smaller than $m$ , so that $\binom{\gamma}{m}=0$.

\medskip

Now we state a lemma that bounds the $2$-adic valuation of a binomial coefficient
$\binom{\gamma}{m}$, with $\gamma\in\Z_2$ and $m\in\N$, in terms of the length of the binary expansion of $m$ and the position of the next non-zero digit in the binary expansion of $\gamma$.

\begin{lem}\label{lem:2valuation}
Let $(a_n)$ be an increasing sequence of non-negative integers and let $\gamma=\sum_{i=0}^\infty 2^{a_{i}}$.
Let $m=\sum_{i=0}^{N} 2^{b_{i}}$ be a positive integer with its base-2 expansion, and $M$ the integer such that
\[
 \sum_{i=0}^{M-1} 2^{a_{i}} < m \leq \sum_{i=0}^{M} 2^{a_{i}},
\]
then $v_2\left(\binom{\gamma}{m}\right)\geq a_M-b_N$.
If in particular $m=2^{b}$ is a power of 2, then $v_2\left(\binom{\gamma}{2^b}\right)= a_M-b$.
\end{lem}
\begin{proof}
 The statement follows directly from Kummer's theorem, because in the addition between $m$ and $\gamma -m$ the digits between the $b_N$-th and the $a_M$-th will have a carry.
\end{proof}

Now we apply this lemma to two $2$-adic integers of a suitable shape.

\begin{prop}\label{prop:2val}
 Let $(a_n)$ be an increasing sequence of non-negative integers such that $\lim_{n\to\infty}(a_{n+1}-a_n)=+\infty$.
Let $\gamma_1=\sum_{i=0}^\infty 2^{a_{2i}}$ and $\gamma_2=\sum_{i=0}^\infty 2^{a_{2i+1}}$.

Then $\binom{\gamma_1}{m}\binom{\gamma_2}{m}\to 0$ 2-adically; for  $m=2^{a_j}$ we have $v_2\left(\binom{\gamma_1}{2^{a_j}}\binom{\gamma_2}{2^{a_j}}\right) = a_{j+1}-a_j.$

\end{prop}
\begin{proof}
Let $m=\sum_{i=0}^{N} 2^{b_{i}}$ be the base-2 expansion of a positive integer $m$, and $M$ the integer such that
\[
 \sum_{i=0}^{M-1} 2^{a_{i}} < m \leq \sum_{i=0}^{M} 2^{a_{i}}.
\]
Applying Lemma~\ref{lem:2valuation} to both $\gamma_1$ and $\gamma_2$ we see that $$v_2\left(\binom{\gamma_1}{m}\binom{\gamma_2}{m}\right)\geq a_{M}-b_N+a_{M+1}-b_N =a_{M+1}-a_M +2(a_M-b_N)\geq a_{M+1}-a_M.$$
As $m\to\infty$ so do $N,b_N$ and $M$, and therefore $v_2\left(\binom{\gamma_1}{m}\binom{\gamma_2}{m}\right)\to\infty$.

Furthermore, if $m=2^{a_j}$, then $b_N=a_M=a_j$ and
$$v_2\left(\binom{\gamma_1}{2^{a_j}}\binom{\gamma_2}{2^{a_j}}\right)= a_{j}-a_j+a_{j+1}-a_j = a_{j+1}-a_j$$
again by Lemma~\ref{lem:2valuation}.
\end{proof}

We can now give an explicit choice of the $\gamma_1,\gamma_2$.

Take for instance $a_n=n^2$ and define the  $2$-adic integers $\gamma_1=\sum_{i=0}^\infty 2^{(2i)^2}, \gamma_2=\sum_{i=0}^\infty 2^{(2i+1)^2}$ as in the previous proposition.

Let $m=\sum_{i=0}^{N} 2^{b_{i}}$ be the base-2 expansion of a positive integer $m$; then we see that
$b_N= \floor{\log_2 m}, M-1 \leq \sqrt{b_N}\leq M$ and $a_{M+1}-a_M=2M+1$, so that applying the proposition we obtain  $v_2\left(\binom{\gamma_1}{m}\binom{\gamma_2}{m}\right)\gg (\log m)^{1/2}$.

This shows that the series
\[
 s(x)=\sum_{m=0}^\infty \binom{\gamma_1}{m}\binom{\gamma_2}{m} x^m
\]
converges for $\abs{x}_2\leq 1$.
Furthermore, if $m=2^{n^2}$ then $v_2\left(\binom{\gamma_1}{m}\binom{\gamma_2}{m}\right)=2n+1$, which shows that the series does not converge for $\abs{x}_2>1$.

\begin{rem}
 For every $\gamma\in\Z_2\setminus \Z$ the sequence $\binom{2m}{m}\binom{\gamma}{m}$ does not tend to 0 2-adically.
Indeed if $m=2^n$ is a power of 2 which appears in the base-2 expansion of $\gamma$, then $v_2\left(\binom{\gamma}{m}\right)=0$ and $v_2\left(\binom{2^{n+1}}{2^n}\right)=1$.

This shows that it was not possible to construct a similar example using a single  $\gamma$, even with the help of an additional binomial coefficient. See Example~\ref{EX.2} for a more general discussion.
\end{rem}

Concerning all the powers $s(x)^q$ of  the specific $s(x)$ constructed above, it would be possible to prove the stated stronger conclusion that  their  radius of convergence is  $1$,
by analysing carefully the $2$-adic valuation of sums of products of binomial coefficients. Instead, it is simpler to  prove a more general lemma that avoids intricate combinatorial arguments and applies to a very general class of series.

\begin{rem}\label{rem:radiusge1}
 A series $\sum_{n\geq 0}a_n x^n \in 2\Z_2[[x]]$ has a radius of convergence greater than 1 if and only if there exists an $\alpha>0$ such that  {\it for all } $n\geq 0$ we have $v_2(a_n)\geq \alpha n$.

Indeed, if the radius of convergence is $>1$ there exists an $\alpha$ such that the inequality is true for all large $n$.  But then the assumption that all $a_n$ are divisible by 2 guarantees that we can shrink $\alpha$ in such a way as to satisfy it for \emph{all} $n\in\N$.
\end{rem}

\begin{lem}\label{lem:radius_powers}
 Let $s(x)=1+4\delta(x)$, where  $\delta(x)\in\Z_2[[x]]$  is a series such that $\delta(0)=0$.
 If $s(x)$ has a radius of convergence equal to 1, the same holds for all its powers.
\end{lem}
\begin{proof}
 It is enough to show this for exponents that are powers of 2 (since the series $(1+x)^{1/m}$ has coefficients in $\Z_2$ for odd $m$).
We can write
 \[
  S(x):=s(x)^{2^k}=1+2^{k+1}\Delta(x)
 \]
with $\Delta\in 2\Z_2[[x]]$. To see this, expand $s(x)^{2^k}=1+\sum_{m=1}^{2^k}\binom{2^k}{m}4^m\delta(x)^m$ and notice that, by the result of Kummer already mentioned, for all $m\geq 1$
$v_2\left(\binom{2^k}{m}\right)+2m\geq k+2$.

Assume that $S(x)$ has a radius of convergence greater than 1. Then by  Remark~\ref{rem:radiusge1}, writing $\Delta(x)=\sum_{n\ge 0}\Delta_nx^n$,  there exists $\alpha>0$ such that $v_2(\Delta_n)\geq \alpha n$ for all $n\ge 0$.
The same holds for each power $\Delta(x)^m$, as the inequality is linear homogeneous.

But now we can recover $s(x)$ from $S(x)$ again by the binomial theorem, and obtain
\[
s(x)=\sum_{m\geq 0}\binom{1/2^k}{m}(2^{k+1}\Delta(x))^m.
\]
The product $\binom{1/2^k}{m}2^{(k+1)m}$ lies in $\Z_2$, and the valuation of the coefficients of $x^n$ in each power $\Delta(x)^m$ is at least $\alpha n$, which shows that the whole series $s(x)$ has a radius of convergence greater than 1, which is a contradiction.
\end{proof}
The lemma is ultimately related to the fact that the coefficients of the Taylor series of $\sqrt{1+4x}$ are integers, and can also be obtained from the $2$-adic exponential and logarithm series.

\begin{proof}[Proof of Theorem~\ref{T.N2}]
Now to conclude the proof of Theorem~\ref{T.N2} we choose an increasing sequence $a_n$ such that $a_0=0, a_1\geq 2$ and $a_{n+1}-a_n\asymp n$ and define $\gamma_1,\gamma_2,s(x)$ as above. Proposition~\ref{prop:2val} and Lemma~\ref{lem:radius_powers} show that the radius of convergence of $s(x)$ and all its powers is exactly one.
\end{proof}

\begin{rem}\label{rem:ex_ogni_raggio}
 Using the same construction as in Proposition~\ref{prop:2val} it is possible to construct analogous series with any radius of convergence greater than 1.

 Fixing $\beta$ a positive real number, if $\beta=\liminf_{m\to\infty} \frac{1}{m}v_2\left(\binom{\gamma_1}{m}\binom{\gamma_2}{m}\right)$ then the radius of convergence of the series $s(x)$ is exactly $2^\beta$.  This can be achieved by defining the sequence $a_n$ inductively as $a_0=1, a_{n+1}=a_n+\ceiling{\beta 2^{a_n}}$ and applying Lemma~\ref{lem:2valuation} as in the proof of Proposition~\ref{prop:2val}.

 Arguing around the Newton polygon of this series one can also show that it has infinitely many zeroes in its disk of convergence.

 This shows that the rationality of the logarithm in base $p$ of the radius of convergence for algebraic series cannot be reduced to a question about $p$-adic differential equations alone (see \cite{Kedlaya_2010} for related discussions).
\end{rem}

\section{Linear differential equations of order 1: analogues of Theorem~\ref{T.BDR}}\label{R.N2}

As mentioned after the statement of Theorem~\ref{T.N2}, one can ask what happens for  even simpler examples   involving differential equations $s'(x)=u(x)s(x)$ of order $1$, where $u\in K(x)$.   Again, the issue itself concerns only  the ultrametric case. In this remark we discuss in short this question, and sketch a proof of an analogue of Theorem~\ref{T.BDR} for series of this kind.

 If the equation is Fuchsian, i.e. has only regular singularities, then the finite poles of $u(x)$ have  order $1$ and $u$ vanishes at $\infty$. By subtracting $m/x$ from $u$ ($m\in\N$)   we can also assume that $0$ is not a pole of $u$ (i.e. $s(0)\neq 0$). 
Then, an integral $s(x)$ is
a constant times   a finite product  of  functions the shape $(1+bx)^{\beta}=\sum  \binom{\beta}{m}b^mx^m$, where $b,\beta\in K^*$.
 
 \medskip
 
 Now, suppose  that  the series for $s(x)$ converges for $|x|\le 1$.  Then it has only finitely many zeros in $D(0,1^+)$, as is well-known (see \cite{Am} or \cite{DGS}). Moreover, if $x_0\in D(0,1^+)$ is a zero, we may divide out by $x-x_0$ and replace $s(x)$ by $s_1(x):=(s(x)-s(x_0))/(x-x_0)$, which has the effect to replace $u(x)$ by $u(x)-(x-x_0)^{-1}$.   Note (as in Remark~\ref{T.BDR:NOTE}) that the new series $s_1(x)$ again converges  on $|x|\le 1$.  Then, after iterating  this operation a finite number of times, we may suppose that $s(x)$ has no zeros in $D(0,1^+)$.
 
 At this stage we observe that $u(x)$ cannot have poles $x_0\in D(0,1^+)$; in fact, otherwise we could write $(x-x_0)s'(x)=((x-x_0)u(x))\cdot s(x)$ and evaluate both sides at $x=x_0$ (because of convergence on $|x|=1$), obtaining $s(x_0)=0$, which we have excluded.   
 Hence, dividing out $s(x)$ by a suitable polynomial, we may assume that  all the poles of $u(x)$ have absolute value $>1$, hence a series for $u(x)$ has radius of convergence $>1$.
 
\begin{example}\label{EX.2}  In particular, taking $s(x)=(1+x)^\beta$, this proves that no binomial sequence $\binom{\beta}{m}$, $\beta\in K$, can converge $p$-adically to $0$ as $m\to\infty$, unless $\beta$ is a natural number: it suffices to apply the above reduction steps starting with    $u(x):=s'(x)/s(x)=\beta/(1+x)$. Of course there are also simple direct arguments for this;  however that does not seem equally simple for slightly more involved cases. For instance taking $(1+ax)^{\alpha}(1+bx)^{\beta}$, with $p$-adic integers $\alpha,\beta\in\Z_p$ not both in $\N$,  and  distinct  $p$-adic units $a\neq b\in\Z_p^*$, we obtain that the sequence  $(c_m)$ defined by $c_m:=\sum_{r+s=m}a^rb^s \binom{\alpha}{r} \binom{\beta}{s}$ cannot converge to zero in $\Z_p$.  This can be compared with Theorem \ref{T.N2}.
\end{example}
 
 The fact that $s(x)$ has no zeros in $D(0,1^+)$ also implies, looking at Newton polygons,  that, after dividing by a constant, we may assume that $s(0)=1$,  that $s\in\O[[x]]$, and that $|s(x)-1|_G<1$ (in the sense of the Gauss norm).  For this see \cite{Am} or \cite{DGS}.   
 
 Now write $u(x)$ as a  finite sum  of terms $\beta_i(1-b_ix)^{-1}$, where  $b_i,\beta_i\in K^*$ and where $0<|b_i|<1$. Also let $q=p^r$ be a power of $p$ so large that $q\beta_i\in \O$ for all $i$.  Then the opening discussion,  on  integrating $s(x)$ explicitly,   shows that $s(x)^q$ has radius of convergence $>1$. 
 
 So, we have reached the same conclusion of Theorem \ref{T.BDR}, however  at the cost of raising $s(x)$ to a positive power.  Note that this last condition is relevant, as shown by examples like $(1+p^{p/(p-1)}x)^{1/p}$, where the $p$-th power has infinite radius of convergence whereas the binomial expansion $\sum_{m=0}^\infty \binom{1/p}{m}p^{mp/(p-1)}x^m$ has radius of convergence $1$ and does not converge in the whole   closed disk $D(0,1^+)$.

This already  shows a rather strong  limitation in order to produce examples like in Theorem \ref{T.N2},  but with differential equations of order $1$.
In fact, we now prove  this assertion
 in a strong sense.
 
\begin{prop} \label{P.N2} For a solution $s(x)$ as above the conclusion of Theorem \ref{T.BDR} holds, i.e. if $s(x)$ converges in $D(0,1^+)$, then in fact it has radius of convergence $>1$.
\end{prop}

\begin{proof}[Sketch of proof.] We work under the reduction steps explained above, namely assuming that $s(0)=1$ and that $|s(x)-1|_G<1$, which means that for $m>0$ the coefficients $a_m$ lie in the maximal ideal $\mathfrak{M}$. Since they tend to $0$ by assumption, we may then find a polynomial $g(x)\in\O[x]$, $g(x)\equiv 1\pmod{\mathfrak{M}}$ such that $g(0)=1$ and $|s(x)g(x)-1|_G<p^{-2}$ (just truncate a series for $s(x)^{-1}$).  Note that the roots of $g$ will have absolute value $>1$, hence $g^{-1}$ has radius of convergence $>1$, and thus it suffices to prove the result after replacing $s$ with $sg$, and thus assuming that $|s(x)-1|_G<p^{-2}$.

With this proviso, setting $\theta=\theta(x):=s(x)-1$, we have   that $|\theta|_G<p^{-2}$.
Also, letting $q=p^r>1$ be  such that the conclusion holds for $s^q$,  we have $$s(x)^q=(1+\theta)^q=1+q\theta+\binom{q}{2}\theta^2+\dotsb +\theta^q.$$ It is easy to check that the maximal Gauss valuation of the terms after the first one,  is attained at the second term, hence putting $s(x)^q=1+\phi$ we have $|\phi|_G\le |q\theta|_G< (p^2q)^{-1}$. But then  the formal series expansion for $$\log s=\log(1+\phi)^{1/q}=l(x):=q^{-1}\log(1+\phi)=q^{-1}(\phi-\phi^2/2+\dotsb)$$ holds inside $\O[[x]]$.

 Note that $1+\phi=s(x)^q=\prod (1-b_ix)^{q\beta_i}$ has radius of convergence  $>1$,  hence $\phi$ has radius of convergence $>1$ as well. Then we may write $\phi=p^2q\psi$, where $\psi\in \O[[x]]$ has radius of convergence say $>R>1$. 

The next step is to bound $\psi$ on some disk of radius $>1$. To this end, write $\psi(x)=\sum_{m=0}^\infty \gamma_mx^m$ with $\gamma_m\in\O$, so they satisfy $|\gamma_m|\le \min (1, CR^{-m})$ for a suitable $C>1$.  Define $m_0\ge 0$ as the integer such that $CR^{-m_0-1}<1\le CR^{-m_0}$.  Also, for given small $\epsilon >0$ find $R_0>1$, $R_0<R$,  so close to $1$  that $\max (R_0^{m_0},C(R_0/R)^{m_0+1})<1+\epsilon$, which certainly we may do as the maximum tends to $1$ for $R_0\to 1^+$.  Note that  then for $|x|\le R_0$ the series $\sum \gamma_mx^m$ for $\psi(x)$ converges and we have $|\psi(x)|\le  \max(|\gamma_mx^m|)\le 1+\epsilon$.

Hence, letting $\Delta=D(0,R_0^+)$, we have that    $\phi$ converges in $\Delta$ and takes therein values  whose absolute value is $<(p^2q)^{-1}(1+\epsilon)$. It follows that  the series  $l(x)$ satisfies $|l(x)|_G\le p^{-2}$ and   on $\Delta$   takes values in the disk $D(0,p^{-2}(1+\epsilon)^+)$ (for small enough $\epsilon$). In turn, we may write the formal equality
$$s(x)=\exp\log s=\exp  l(x)=1+l(x)+(l(x))^2/2!+\dotsb ,$$
because $|l|_G\le p^{-2}<p^{-1/(p-1)}$.  If $\epsilon$ is small enough so that $p^{-2}(1+\epsilon)<p^{-1/(p-1)}$ the equality remains valid for $x\in\Delta$, and we also conclude that the  radius of convergence of $s(x)$ is $\ge R_0>1$, as wanted. (See \cite{DGS}, especially Lemma 7.1,  for these  last deductions.)
\end{proof}

Since the matter is not quite near the rest of the paper, we have not pushed forward the analysis for differential equations of order 1 which are non-Fuchsian, though we are inclined to think that the methods prove that the result still holds. In that case one should be able to reduce to series of the shape $\exp (c_1x+c_2x^2+\dotsb)$, where $c_i\to 0$. It should be not too difficult  to prove that such series cannot converge on the closed disk of convergence. (See indeed the paper \cite{BS} of Bombieri and Sperber for much more general conclusions, and see  also Crew's paper \cite{C}, \S~4, especially  Prop.~4.4 and  Prop.~4.5, for results quite related to the above.)

\end{document}